\documentclass[12pt,a4paper]{article}

\usepackage{amsmath,amssymb}
\usepackage{amsthm}

\newtheorem{proposition}{Proposition}

\newtheorem{theorem}{Theorem}

\title{Periodic linear discrete multitime diagonal recurrence with application in economics}

\author{Cristian Ghiu, Raluca Tulig\u{a}, \\Constantin Udri\c ste, Ionel \c Tevy}
\date{}

\begin{document}

\maketitle

\begin{abstract}
Floquet theory, first published in $1883$ for periodic linear
differential equations, is extended in this paper to multitime
diagonal recurrences. We find explicitly a monodromy matrix, and
we comment its eigenvalues (called Floquet multipliers). The
Floquet point of view brings about an important simplification:
the initial linear diagonal recurrence system is reduced to
a linear recurrence system, with constant coefficients along ``diagonal lines".
The results are applied to the discrete multitime Samuelson-Hicks models, with constant,
respectively multi-periodic, coefficients, in order to find bivariate sequences with economic meaning.
For constant coefficients case, it was found also the generating function;
for multi-periodic coefficients case, we determined the Floquet multipliers.
\end{abstract}




\section{Discrete multitime recurrences}

The multivariate recurrences are based on multiple sequences and come from areas like analysis of algorithms,
computational biology, information theory, queueing theory, filters theory, statistical physics etc.

We consider the lattice of points with integer positive coordinates in $\mathbb{R}^n$.
A multi-variate recurrence is a set of rules which transfer a point into another, together
with initial conditions, capable to cover the hole lattice.

A way to study a multiple sequence is to apply a transform (generating
function, Mellin transform etc), yielding a multi-variable complex function $F$.
We then study the singularities of $F$, its Taylor development around origin, and extract information
on the original multiple sequence via multi-variable complex analysis
(especially the explicit formula of the multivariate sequence).
The advantages of generating functions (GFs) are: provide a compact representation of the multivariate sequence,
are well adapted for computer algebra manipulation (symbolic
computation avoids numerical errors and can yield insight);
can prove identities and simplify sums;
can solve recurrences (difference equations);
can yield recurrences, and hence fast computation algorithms;
allow extraction of asymptotic approximations. The multivariate GF is the best all-round tool
for studying multi-variate recurrences.

Bousquet-M\' elou and Petkov\v sek \cite{BP} showed that while in the univariate case solutions
of linear recurrences with constant coefficients
have rational generating functions (see also \cite{E}), the multivariate case is much richer,
even though initial conditions have rational generating functions the corresponding
solutions can have generating functions which are not rational, algebraic, or
even $D$-finite GF in two or more variables.

Pemantle and Wilson \cite{W} describe a symbolic method for computing generating functions for multivariate recurrences.

Hauser and Koutschan \cite{HK} give theorems of division for formal power series, which clarifies the structural
background and provides short, conceptual proofs for generating functions. In addition, extending the division to
the context of differential operators, the case of recurrences with polynomial coefficients
can be treated in an analogous way.

Floquet theory, first formulated for periodic linear ODEs \cite{Floquet},
was extended to difference equations
(\cite{markley}, \cite{yakubovich}, \cite{M}, \cite{K}).
We have extended this theory to the multitemporal
first order PDEs \cite{U} and now to multivariate multi-periodic recurrences, borrowing mathematical ingredients from some papers
\cite{U1}-\cite{U7} involving multitime (multivariate time) dynamical systems. In Floquet theory it is necessary to find
explicitly the associated monodromy matrix and its eigenvalues (called Floquet multipliers).

\section{Linear discrete multitime diagonal \\recurrence with periodic coefficients}

An element $t=(t^1,\ldots,t^m) \in \mathbb{N}^m$ is called {\it discrete multitime}.
A function of the type $x:\mathbb{N}^m \to \mathbb{R}^n$ is called {\it multivariate sequence}.
Also, for convenience, we denote $\mu(t)=\min \{t^1,t^2,\ldots, t^m\}$
and ${\bf 1}=(1,1,\ldots,1) \in \mathbb{N}^m$.

Let $m\geq2$ and $A \colon \mathbb{N}^m \to \mathcal{M}_{n}(\mathbb{R})$.
Our aim is to continue the study of a linear discrete multitime diagonal recurrence system
(see \cite{GTU})
\begin{equation}
\label{linomog}
x(t+{\bf 1})=A(t)x(t),
\quad
\forall
t \in \mathbb{N}^m.
\end{equation}

In the paper \cite{GTU} one proves the next result.

\begin{theorem}
\label{rec.t1}
Let $m\geq 2$,
$A \colon \mathbb{N}^m \to \mathcal{M}_{n}(\mathbb{R})$,
$b \colon \mathbb{N}^m \to \mathbb{R}^n$. We consider the $(m-1)$-sequences
$f_1, f_2, \ldots, f_m \colon \mathbb{N}^{m-1} \to \mathbb{R}^n$,
such that
\begin{equation}
\label{ect1.1}
f_{\alpha}(t^1,\ldots, \widehat{t^\alpha}, \ldots, t^m)\Big|_{t^\beta=0}
=
f_{\beta}(t^1,\ldots, \widehat{t^\beta}, \ldots, t^m)\Big|_{t^\alpha=0},
\end{equation}
\begin{equation*}
\forall
t^1, \ldots, t^{\alpha-1}, t^{\alpha+1}, \ldots, t^{\beta-1}, t^{\beta+1}, \ldots, t^m
\in \mathbb{N},
\end{equation*}
for any $\alpha,\beta \in \{1, 2, \ldots, m\}$.
Then the unique $m$-sequence
$x \colon \mathbb{N}^m \to \mathbb{R}^n$ which verifies
\begin{equation*}
x(t+{\bf 1})
=
A(t)x(t)+b(t),
\quad \forall t \in \mathbb{N}^m,
\end{equation*}
\begin{equation*}
x(t)\Big|_{t^\beta=0}
=
f_{\beta}(t^1,\ldots, \widehat{t^\beta}, \ldots, t^m),
\quad
\forall
(t^1,\ldots, \widehat{t^\beta}, \ldots, t^m)\in \mathbb{N}^{m-1},
\end{equation*}
\begin{equation*}
\forall
\beta \in \{1,2, \ldots, m\},
\end{equation*}
is defined either by the formula
\begin{equation*}
\begin{split}
&x(t)=
A(t-{\bf 1})A(t-2\cdot {\bf 1})\cdot \ldots
\cdot A(t-t^\beta \cdot  {\bf 1})\cdot\\
& \qquad \cdot
f_{\beta}(t^1-t^\beta, \ldots, \widehat{t^\beta}, \ldots, t^{m-1}-t^\beta)
+b(t-{\bf 1})\\
&+
\sum_{k=2}^{t^\beta}
A(t-{\bf 1})A(t-2\cdot  {\bf 1})\cdot
\ldots \cdot A(t-(k-1)\cdot {\bf 1})b(t-k\cdot {\bf 1}),\\
&\qquad \mbox{ if }\,\,
\mu(t)=t^\beta \geq 2,
\end{split}
\end{equation*}
or by the formula
\begin{equation*}
\begin{split}
x(t)
&=
A(t^1-1, \dots ,t^{\beta-1}-1,0,t^{\beta+1}-1, \ldots t^m-1)
\cdot
\\
& \,\, \cdot
f_{\beta}(t^1-1, \ldots, \widehat{t^\beta}, \ldots, t^{m-1}-1)\\
&+b(t^1-1, \dots ,t^{\beta-1}-1,0,t^{\beta+1}-1, \ldots t^m-1),\\
&\quad \quad  \mbox{ if }\,\,
\mu(t)=t^\beta=1.
\end{split}
\end{equation*}
\hfill$\square$
\end{theorem}

The function
\begin{equation*}
\Phi \colon \mathbb{N}^m \to \mathcal{M}_{n}(\mathbb{R}),
\quad
\Phi(t)=
\left\{
\begin{array}{ccc}
\displaystyle \prod_{k=1}^{\mu(t)}A(t-k \cdot  {\bf 1}),& \hbox{ if }& \mu(t) \geq 1, \\
I_n, & \hbox{ if }& \mu(t)=0,
\end{array}
\right.
\end{equation*}
is called the {\it fundamental matrix} ({\it transfer matrix}) associated to the recurrence (\ref{linomog}).
In fact, $\Phi(\cdot)$ is the unique matrix function which verifies the problem
\begin{equation}
\label{fundam}
\begin{split}
\Phi(t+{\bf 1})
&=
A(t)\Phi(t),
\quad \forall t \in \mathbb{N}^m,\\
\Phi(t)\Big|_{t^\beta=0}
&=
I_n,
\quad
\forall
(t^1,\ldots, \widehat{t^\beta}, \ldots, t^m)\in \mathbb{N}^{m-1},\\
\forall
\beta&\in \{1,2, \ldots, m\}.
\end{split}
\end{equation}
This follows applying the Theorem \ref{rec.t1}
for $n$ (vector) recurrences to whom is equivalent the problem
(\ref{fundam}) (one applies the Theorem \ref{rec.t1}
for each column of the matrix $\Phi(\cdot)$).

If the functions
$$
f_1, f_2, \ldots, f_m \colon \mathbb{N}^{m-1} \to \mathbb{R}^n
$$
verify the relations (\ref{ect1.1}),
then one observes (Theorem  \ref{rec.t1})
that the unique
$m$-sequence $x \colon \mathbb{N}^m \to \mathbb{R}^n$ which verifies
the recurrence (\ref{linomog}) and
\begin{equation*}
\begin{split}
x(t)\Big|_{t^\beta=0}
&=
f_j(t^1,\ldots, \widehat{t^\beta}, \ldots, t^m),
\quad
\forall
(t^1,\ldots, \widehat{t^\beta}, \ldots, t^m)\in \mathbb{N}^{m-1},\\
\forall
\beta&\in \{1,2, \ldots, m\},
\end{split}
\end{equation*}
can be written in the form
\begin{equation*}
x(t)=\Phi(t)
f_{\beta}(t^1-t^\beta, \ldots, \widehat{t^\beta}, \ldots, t^{m-1}-t^\beta),
\quad
\mbox{ if }\,\,
\mu(t)=t^\beta.
\end{equation*}

\vspace{0.2 cm}
If the matrix function $A(\cdot)$ is constant,
i.e., $A(t)=A$, $\forall t$, then the fundamental matrix becomes
$\displaystyle \Phi(t)=A^{\mu(t)}$.

A linear discrete multitime diagonal recurrence is called $T$-{\it diagonal-periodic}
($T\in \mathbb{N}^*$) if
\begin{equation} \label{multi} A(t+T\cdot {\bf 1})=A(t),\,\,\forall t\in \mathbb{N}^m.\end{equation}
This is the only multi-periodicity compatible to the diagonal recurrence, independently of the initial conditions.

\begin{proposition}
\label{rec.p3}
Suppose the recurrence {\rm (\ref{linomog})} is $T$-diagonal-periodic and we introduce
the function
\begin{equation*}
\widetilde{A} \colon \mathbb{N}^m \to \mathcal{M}_{n}(\mathbb{R}),
\quad
\widetilde{A}(t)=
\displaystyle \prod_{k=0}^{T-1}
A\big( t+(T-1-k)\cdot {\bf 1}\big),
\quad
\forall t \in \mathbb{N}^m.
\end{equation*}

Then, for any $k \in \mathbb{N}^*$, the fundamental matrix satisfies
\begin{equation}
\label{fiperiod}
\Phi\big( t+kT\cdot {\bf 1}\big)
=
\Phi(t)
\Big(
\widetilde{A}
\big( t-\mu(t) \cdot {\bf 1}\big)
\Big)^k.
\end{equation}
\end{proposition}

\begin{proof}
If $t^\beta=\mu(t)$, then
$$t^\beta+kT=\min\{t^1+kT,t^2+kT,\ldots, t^m+kT\}.$$

Induction after $k$.
For $k=1$:

The case $t^\beta=0$:
$$
\Phi\big( t+T\cdot {\bf 1}\big)
=
A\big( t+(T-1)\cdot {\bf 1}\big)
A\big( t+(T-2)\cdot {\bf 1}\big)
\cdot \ldots \cdot
A\big( t+(T-T)\cdot {\bf 1}\big)
$$
$$
=
I_n\widetilde{A}(t)
=
\Phi(t)\widetilde{A}(t).
$$

Case $t^\beta\geq 1$:
$$
\Phi\big( t+T\cdot {\bf 1}\big)
=
A\big( t+T\cdot {\bf 1}- {\bf 1}\big)
A\big( t+T\cdot {\bf 1}-2\cdot {\bf 1}\big)
\cdot \ldots \cdot
A\big( t+T\cdot {\bf 1}-(t^\beta+T)\cdot {\bf 1}\big)
$$
$$
=
A\big( t+T\cdot {\bf 1}- {\bf 1}\big)
A\big( t+T\cdot {\bf 1}-2\cdot {\bf 1}\big)
\cdot \ldots \cdot
A\big( t+T\cdot {\bf 1}-t^\beta\cdot {\bf 1}\big)\cdot
$$
$$
\cdot
A\big( t-t^\beta\cdot {\bf 1}+(T-1)\cdot {\bf 1}\big)
A\big( t-t^\beta\cdot {\bf 1}+(T-2)\cdot {\bf 1}\big)
\cdot \ldots \cdot
A\big( t-t^\beta\cdot {\bf 1}+{\bf 1}\big)
A\big( t-t^\beta\cdot {\bf 1}\big)
$$
$$
=
A\big( t- {\bf 1}\big)
A\big( t-2\cdot {\bf 1}\big)
\cdot \ldots \cdot
A\big( t-t^\beta\cdot {\bf 1}\big)\cdot
$$
$$
\cdot
A\big( t-t^\beta\cdot {\bf 1}+(T-1)\cdot {\bf 1}\big)
A\big( t-t^\beta\cdot {\bf 1}+(T-2)\cdot {\bf 1}\big)
\cdot \ldots \cdot
A\big( t-t^\beta\cdot {\bf 1}+{\bf 1}\big)
A\big( t-t^\beta\cdot {\bf 1}\big)
$$
$$=\Phi(t)\widetilde{A}
\big( t-t^\beta\cdot {\bf 1}\big).
$$

Suppose the relation (\ref{fiperiod}) is true for any $k$, and we shall verify for $k+1$:
$$
\Phi\big( t+(k+1)T\cdot {\bf 1}\big)
=
\Phi\big( t+kT\cdot {\bf 1}+T\cdot {\bf 1}\big)
=
\Phi(t+kT\cdot {\bf 1})
\widetilde{A}
\big( t+kT\cdot {\bf 1}-(t^\beta+kT)\cdot {\bf 1}\big)
$$
$$
=
\Phi(t)
\Big(
\widetilde{A}
\big( t-t^\beta\cdot {\bf 1}\big)
\Big)^k
\widetilde{A}
\big( t-t^\beta \cdot {\bf 1}\big)
=
\Phi(t)
\Big(
\widetilde{A}
\big( t-t^\beta\cdot {\bf 1}\big)
\Big)^{k+1}.
$$
\end{proof}

Suppose that we are in the conditions of the Proposition \ref{rec.p3}.
The matrix function
\begin{equation*}
D \colon \mathbb{N}^m \to \mathcal{M}_{n}(\mathbb{R}),
\quad
D(t)=
\widetilde{A}\big( t-\mu(t) \cdot {\bf 1}\big),
\quad
\forall t\in \mathbb{N}^m,
\end{equation*}
is called the {\it monodromy matrix} associated to the
$T$-diagonal-periodic recurrence {\rm (\ref{linomog})}.
For $k=1$, the formula {\rm (\ref{fiperiod})} can be written
\begin{equation}
\label{fiCt}
\Phi\big( t+T\cdot {\bf 1}\big)
=
\Phi(t)D(t),
\quad
\forall t \in \mathbb{N}^m.
\end{equation}

We observe that the relation
$D\big( t+ {\bf 1}\big)
=D(t)$,\, $\forall t \in \mathbb{N}^m$
holds.

Moreover, let us suppose that, for any $t \in \mathbb{N}^m$, the matrix
$A(t)$ is invertible, hence $\widetilde{A}(t)$ is also invertible;
it follows that for each $t \in \mathbb{N}^m$,
there exists $\widetilde{B}(t) \in \mathcal{M}_{n}(\mathbb{C})$
(which is not unique), such that
$\displaystyle \widetilde{B}(t)^T=\widetilde{A}(t)$.
For each $t$, we fix such a matrix $\widetilde{B}(t)$;
obviously the matrix $\widetilde{B}(t)$ is invertible.

Define the function
\begin{equation}
\label{defB}
B \colon \mathbb{N}^m \to \mathcal{M}_{n}(\mathbb{C}),
\quad
B(t)=
\widetilde{B}\big( t-\mu(t) \cdot {\bf 1}\big),
\quad
\forall t\in \mathbb{N}^m.
\end{equation}
We observe that the relation
\begin{equation}
\label{Bconst}
B\big( t+ {\bf 1}\big)
=B(t),
\quad
\forall t \in \mathbb{N}^m
\end{equation}
is true, i.e., the matrix $B$ verifies a special recurrence. It follows immediately the relation
\begin{equation}
\label{BtCt}
D(t)=
B(t)^T,
\quad
\forall t \in \mathbb{N}^m.
\end{equation}
Moreover, one observes that if for any $t$, the matrix
$A(t)$ is invertible, then the matrices $\Phi(t)$ and $D(t)$
are invertible too.

\begin{theorem}
\label{rec.t3}
Let $m\geq2$ and $A \colon \mathbb{N}^m \to \mathcal{M}_{n}(\mathbb{R})$,
with $A(t)$ invertible, $\forall t \in \mathbb{N}^m$.
Suppose there exists an integer $T \geq 1$ such that the relation
{\rm (\ref{multi})} is true.

Then there exists a function
$P \colon \mathbb{N}^m \to \mathcal{M}_{n}(\mathbb{C})$,
with the property that
$P(t+T\cdot {\bf 1})=P(t)$, $\forall t$, and such that
the fundamental matrix of the $T$-multi-periodic recurrence
{\rm (\ref{linomog})} is written
\begin{equation}
\label{Pt}
\Phi(t)=P(t)B(t)^{\mu(t)},
\quad
\forall t \in \mathbb{N}^m
\end{equation}
$\big($where $B(\cdot)$ is the function definite by the formula {\rm (\ref{defB})}$\big)$.
\end{theorem}

\begin{proof}
Let $$P \colon \mathbb{N}^m \to \mathcal{M}_{n}(\mathbb{C}),$$
$$\displaystyle
P(t)=\Phi(t)B(t)^{-\min\{t^1,t^2,\ldots, t^m\}},\,\,
\forall t \in \mathbb{N}^m.$$
It is sufficient to show
$P(t+T\cdot {\bf 1})=P(t)$, $\forall t \in \mathbb{N}^m$.

Let $t=(t^1,t^2,\ldots, t^m)\in \mathbb{N}^m$
and $t^\beta=\mu(t)$. Obviously,
$$t^\beta+T=\min\{t^1+T,t^2+T,\ldots, t^m+T\},$$
$$P(t+T\cdot {\bf 1})=\Phi(t+T\cdot {\bf 1})B(t+T\cdot {\bf 1})^{-t^\beta-T}.$$
From the relation (\ref{Bconst}), we deduce
$B(t+k\cdot {\bf 1})=B(t)$, $\forall k \in \mathbb{N}$,
and particularly $B(t+T\cdot {\bf 1})=B(t)$.
From the relations (\ref{fiCt}), (\ref{BtCt}), it follows
$\Phi(t+T\cdot {\bf 1})=\Phi(t)B(t)^{T}$.
Hence
\begin{equation*}
P(t+T\cdot {\bf 1})
=
\Phi(t)B(t)^{T}B(t)^{-t^\beta-T}
=
\Phi(t)B(t)^{-t^\beta}
=P(t).
\end{equation*}
\end{proof}
In the condition of the Theorem \ref{rec.t3},
from the formula (\ref{Pt}) and the fact that the matrix $\Phi(t)$ is invertible,
it follows that $P(t)$ is also invertible.

The most important result of Floquet type is
\begin{theorem}
\label{rec.t4}
Let $m\geq2$ and $A \colon \mathbb{N}^m \to \mathcal{M}_{n}(\mathbb{R})$,
with $A(t)$ invertible, $\forall t \in \mathbb{N}^m$.
Suppose that there exists an integer $T \geq 1$ such that the relation
{\rm (\ref{multi})} is true.
Let $B\colon \mathbb{N}^m \to \mathcal{M}_{n}(\mathbb{C})$
be the function definite by the formula {\rm (\ref{defB})}.

We consider the recurrences
\begin{equation}
\label{ect4.1}
x(t+{\bf 1})=A(t)x(t),
\quad
\forall t \in \mathbb{N}^m;
\end{equation}
\begin{equation}
\label{ect4.2}
 y(t+{\bf 1})=B(t)y(t),
\quad
\forall t \in \mathbb{N}^m.
\end{equation}

If $y(t)$ is a solution of the recurrence
{\rm (\ref{ect4.2})}, then
$x(t):=P(t)y(t)$ is a solution of the recurrence
{\rm (\ref{ect4.1})}. And conversely, if
$x(t)$ is a solution of the recurrence
{\rm (\ref{ect4.1})}, then
$y(t):=P(t)^{-1}x(t)$ is a solution of the recurrence
{\rm (\ref{ect4.2})}.
\end{theorem}

\begin{proof}
Let $y(t)$ be a solution of the recurrence
{\rm (\ref{ect4.2})} and $x(t):=P(t)y(t)$; hence
$y(t):=P(t)^{-1}x(t)$.
We fix $t\in \mathbb{N}^m$ and $t^\beta=\mu(t)$. It follows
$$
y(t+{\bf 1})=B(t)y(t)
\Longleftrightarrow
P(t+{\bf 1})^{-1}x(t+{\bf 1})=B(t)P(t)^{-1}x(t)
$$
$$
\Longleftrightarrow
x(t+{\bf 1})=P(t+{\bf 1})B(t)P(t)^{-1}x(t).
$$
We use the formula (\ref{Pt}) and deduce that
the foregoing relations are equivalent to
$$
x(t+{\bf 1})=
\Phi(t+{\bf 1})B(t+{\bf 1})^{-t^\beta-1}
B(t)
B(t)^{t^\beta}
\Phi(t)^{-1}x(t).
$$
This relation is equivalent to (according the formula (\ref{Bconst}))
$$
x(t+{\bf 1})=
\Phi(t+{\bf 1})B(t)^{-t^\beta-1}
B(t)
B(t)^{t^\beta}
\Phi(t)^{-1}x(t)
$$
$$
\Longleftrightarrow
x(t+{\bf 1})=
\Phi(t+{\bf 1})
\Phi(t)^{-1}x(t)
\Longleftrightarrow
x(t+{\bf 1})=
A(t)\Phi(t)
\Phi(t)^{-1}x(t)
$$
$$
\Longleftrightarrow
x(t+{\bf 1})=
A(t)x(t).
$$

Like it proves the reverse.
\end{proof}

In this way, the Floquet point of view brings about an important simplification:
the linear diagonal recurrence system {\rm (\ref{ect4.1})}
is reduced to the linear diagonal recurrence system {\rm (\ref{ect4.2})},
with constant coefficients along ``diagonal lines".

\section{Two-time recurrence of way required}

Let $t\in \mathbb{N}^2$. A first order linear bi-recurrence of the form
$$x(t^1+1,0)= F_1((t^1,0),x(t^1,0)),\,\,x(t^1,t^2+1)= F_2((t^1,t^2),x(t^1,t^2)),$$
with initial condition $x(0,0)=x_0$, is called {\it two-time recurrence of way required}.
Of course, similar to this recurrence, we can consider {\it discrete multitime recurrences of way required}
or {\it discrete multitime nonholonomic recurrences}.

Suppose we have linear two-time recurrence of way required, with constant coefficients,
$$x(t^1+1,0)= A_1 x(t^1,0),\,\, x(t^1,t^2+1) =A_2 x(t^1,t^2).$$
\begin{theorem} The solution (imposed by the path) of a linear two-time recurrence of way required, with constant coefficients, is
$$x(t^1,t^2)= A_2^{t^2}\,A_1^{t^1}x_0.$$\end{theorem}

\section{Discrete multitime constant coefficients \\Samuelson-Hicks model}

We assume that $t=(t^1,...,t^m)\in \mathbb{N}^m$ is a discrete multitime. Having in mind the discrete single-time
Samuelson-Hicks model \cite{M}, we introduce a {\it discrete multitime Samuelson-Hicks like model} based on the following ingredients:
(i) two parameters, the first $\gamma$, called the {\it marginal propensity to consume}, subject to $0<\gamma < 1$, and
the second $\alpha$ as {\it decelerator} if $0<\alpha <1$, {\it keeper} if $\alpha =1$ or {\it accelerator} if $\alpha > 1$;
(ii) the multiple sequence $Y(t)$ means the {\it national income} and is the main endogenous variable, the multiple
sequence $C(t)$ is the {\it consumption}; (iii) we assume that multiple sequences $Y(t)$, $C(t)$ are non-negative.

We propose a {\it first order discrete multitime constant coefficients Samuelson-Hicks model} either
as first order diagonal recurrence system
$$\begin{pmatrix}
  Y(t+{\bf 1}) \\
  C(t+{\bf 1})
 \end{pmatrix}=
 \begin{pmatrix}
  \gamma+\alpha & -\frac{\alpha}{\gamma} \\
  \gamma & 0
 \end{pmatrix}
 \begin{pmatrix}
 Y(t) \\
  C(t)
 \end{pmatrix},\,\, t\in \mathbb{N}^m,$$
 \begin{equation}
 \label{h1}
 Y(t)|_{t^\alpha=0}= Y_{0\alpha}(t^1,...,\widehat{t^\alpha},...,t^{m}),\,\, C(t)|_{t^\alpha=0}= C_{0\alpha}(t^1,...,\widehat{t^\alpha},...,t^{m})
 \end{equation}
 \begin{equation}
 \label{h2}
 Y_{0\alpha}(t^1,...,\widehat{t^\alpha},...,t^{m})|_{t^\beta=0}=Y_{0\beta}(t^1,...,\widehat{t^\beta},...,t^{m})|_{t^\alpha=0}
 \end{equation}
 \begin{equation}
 \label{h3}
 C_{0\alpha}(t^1,...,\widehat{t^\alpha},...,t^{m})|_{t^\beta=0}=C_{0\beta}(t^1,...,\widehat{t^\beta},...,t^{m})|_{t^\alpha=0}
 \end{equation}
or as the second order homogeneous diagonal recurrence equation
\begin{equation}
\label{h4}
Y(t+{\bf 2})-(\gamma+\alpha)Y(t+{\bf 1})+\alpha Y(t)=0,\,\, t\in \mathbb{N}^m,
\end{equation}
\begin{equation}\label{h5}
Y(t)|_{t^\alpha=0}= Y_{0\alpha}(t^1,...,\widehat{t^\alpha},...,t^{m}),\,\,Y(t)|_{t^\alpha=1}= Y_{1\alpha}(t^1,...,\widehat{t^\alpha},...,t^{m})
\end{equation}
\begin{equation}\label{h6}
  Y_{0\alpha}(t^1,...,\widehat{t^\alpha},...,t^{m})|_{t^\beta=0}=Y_{0\beta}(t^1,...,\widehat{t^\beta},...,t^{m})|_{t^\alpha=0}
\end{equation}
\begin{equation}\label{h7}
  Y_{1\alpha}(t^1,...,\widehat{t^\alpha},...,t^{m})|_{t^\beta=1}=Y_{1\beta}(t^1,...,\widehat{t^\beta},...,t^{m})|_{t^\alpha=1}
\end{equation}
\begin{equation}\label{h8}
  Y_{0\alpha}(t^1,...,\widehat{t^\alpha},...,t^{m})|_{t^\beta=1}=Y_{1\beta}(t^1,...,\widehat{t^\beta},...,t^{m})|_{t^\alpha=0}.
\end{equation}

\subsection{Samuelson-Hicks discrete multitime recurrence\\ equation with constant coefficients}

\subsubsection{Characteristic equation}

Let us consider the second order homogeneous discrete multitime diagonal recurrence equation (\ref{h4}),
with constant coefficients $\gamma$ and $\alpha$. Looking for a solution of the form
$$Y(t) = v \lambda^{<\epsilon, t>},\,\,<\epsilon,{\bf 1}>=1,$$
it appears the characteristic equation
$$\lambda^{2}- (\gamma+\alpha)\lambda +\alpha =0.$$

The discriminant of this equation is $\delta= (\gamma+\alpha)^2-4 \alpha$.
If $\delta >0$, then we have real roots $\lambda_1>0, \lambda_2>0$;
if $\delta =0$, then $\lambda_1= \lambda_2> 0$; if $\delta<0$, then the roots $\lambda_1, \lambda_2$ are complex
conjugated with modulus $\sqrt{\alpha}$. The stability requires $|\lambda_i|<1,\,i=1,2$.

\subsubsection{Generating function}

{\bf First computation variant} We start with the recurrence
$$Y_{m+2\,n+2} - (\gamma + \alpha) Y_{m+1\,n+1} + \alpha Y_{mn}=0, \,m\geq 0,\, n\geq 0.$$
Our aim is to determine a closed form for the corresponding generating function
$$F(x,y)=\sum_{m=0,n=0}Y_{mn}x^my^n$$
(analytic in a neighborhood of the origin; it can often be
extended by analytic continuation) in terms of the given initial
conditions and the recurrence relation, and then to derive a direct access formula
for the coefficients $Y_{mn}$ of the development of $F(x,y)$ in Taylor series around origin.

Multiplying with $x^{m+2}y^{n+2}$ and summing, it follows
$$\sum_{m,n\geq 0}Y_{m+2\,n+2}x^{m+2}y^{n+2} -(\gamma + \alpha)\sum_{m,n\geq 0}Y_{m+1\,n+1}x^{m+2}y^{n+2}$$
$$ +\alpha \sum_{m,n\geq 0}Y_{mn}x^{m+2}y^{n+2}=0.$$
We replace each double infinite sum in this equation by an algebraic expression involving the generating function, respectively
$$
\sum_{m,n\geq 0}Y_{m+2\,n+2}x^{m+2}y^{n+2}=F(x,y)-\sum_{m\geq 0}Y_{m0}x^m -\sum_{n\geq 0}Y_{0n}y^n
+Y_{00}$$
$$-y \sum_{m\geq 1}Y_{m1}x^m -x \sum_{n\geq 1}Y_{1n}y^n +xy Y_{11},$$
$$\sum_{m,n\geq 0}Y_{m+1\,n+1}x^{m+2}y^{n+2}=xy\left(F(x,y)+Y_{00}-\sum_{m\geq 0}Y_{m0}x^m -\sum_{n\geq 0}Y_{0n}y^n \right)$$
$$\sum_{m,n\geq 0}Y_{mn}x^{m+2}y^{n+2}=x^2y^2F(x,y).$$
We denote
$$\phi_0(x)=\sum_{m\geq 0}Y_{m0}x^m,\,\psi_0(y) =\sum_{n\geq 0}Y_{0n}y^n,$$
$$\phi_1(x)=\sum_{m\geq 1}Y_{m1}x^m,\,\psi_1(y) =\sum_{n\geq 1}Y_{1n}y^n.$$
It follows the functional equation $Q(x,y) F(x,y)=G(x,y)$, which define the generating function
$F(x,y)=\frac{G(x,y)}{Q(x,y)},$ where
$$Q(x,y)= 1-(\gamma + \alpha)xy + \alpha x^2y^2,$$
$$G(x,y)=-(\gamma + \alpha)xy(\phi_0(x)+\psi_0(y)-Y_{00})-xy Y_{11}$$
$$ + y\phi_1(x)+x\psi_1(y)+(\phi_0(x)+\psi_0(y)-Y_{00}).$$

The polynomial $Q(x,y)$ is the transform in $\frac{1}{\lambda}$ of the characteristic polynomial of the recurrence.

We write $G(x,y)= K(x,y)-U(x,y)$, with $$K(x,y)= -(\gamma + \alpha)xy(\phi_0(x)+\psi_0(y)-Y_{00})-xy Y_{11}$$
and
$$U(x,y)=- y\phi_1(x)-x\psi_1(y)-(\phi_0(x)+\psi_0(y)-Y_{00}).$$
Then
$$K(x,y)= F(x,y)G(x,y)+U(x,y),$$
where $K(x,y)$ is a polynomial in $xy$, like $Q$, whose coefficients are functions depending on
the initial conditions. The function $U(x,y)$ is affine, with coefficients depending on the initial conditions.

Using the formal power series expansion
$$F(x,y)= G(x,y) \sum_{k=0}^{\infty} \left((\gamma + \alpha) - \alpha xy\right)^k\,x^ky^k,$$
we find the general term $Y_{mn}$ of the bivariate economic sequence.

{\bf Particular solution} If we use the following data
$$\begin{array}{ccccc}\phi_0(x)=Y_{00}&\Longleftrightarrow  & Y_{m0}=0 &\hbox{for}& m\geq 1\\
\psi_0(y)=Y_{00}&\Longleftrightarrow & Y_{0n}=0& \hbox{for}& n\geq 1\\
\phi_1(x)=x&\Longleftrightarrow & Y_{11}=1, Y_{m1}=0 &\hbox{for} &m\geq 2\\
\psi(y)=y&\Longleftrightarrow & Y_{11}=1, Y_{1n}=0&\hbox{for}& n\geq 2,\end{array}$$
we obtain
$$F(x,y)=\frac{xy(1-(\gamma + \alpha)) + Y_{00}}{1-(\gamma +\alpha)xy +\alpha x^2y^2}.$$
Denoting by $r_1, r_2$ the real roots of the polynomial  $Q(x,y)$, it follows the bivariate sequence
$$Y_{nn}= \frac{1}{\alpha}\left(\frac{A}{r_1^{n+1}}+\frac{B}{r_2^{n+1}}\right)\,\, \hbox{and}\,\,Y_{mn}=0,\,\, \hbox{for}\,\, m\not = n,$$
depending on the parameters $\alpha$ and $\gamma$.

{\bf Second computation variant} Let us start with the univariate Samuelson-Hicks recurrence
$$a_{n+2}-(\alpha+\gamma)a_{n+1}+\alpha a_n=0.$$
Using a method similar to the previous one, we obtain the generating function
$$F(x)=\frac{a_0+(a_1-(\alpha +\gamma)a_0)x}{1-(\alpha + \gamma)x +\alpha x^2},$$
where $a_0$, $a_1$ are fixed by the initial conditions.

For the bivariate recurrence, the general solution, via characteristic equation, is
$$Y_{mn} = c_{1}\lambda_1^{\min\{m,n\}} + c_{2} \lambda_2^{\min \{m,n\}},$$
where the coefficients $c_1$, $c_2$ are fixed by initial conditions.
Introducing the subsets $S_1 =\{ m\geq 0,\, n\geq 0,\,\, m\leq n\}$, $S_2 =\{ m\geq 0,\, n\geq 0,\,\, m > n\}$ of $\mathbb{N}^2$,
the generating function becomes
$$F(x,y)= \sum_{S_1}\left(c_{1}\lambda_1^m + c_{2} \lambda_2^m \right)(xy)^m y^{n-m}
 + \sum_{S_2}\left(c_{1}\lambda_1^n + c_{2} \lambda_2^n \right)(xy)^n x^{m-n}$$
$$= \sum_{k\geq 0}\left[\sum_{m\geq 0}\left(c_{1}\lambda_1^m + c_{2} \lambda_2^m \right)(xy)^m\right]y^k
+ \sum_{k\geq 1}\left[\sum_{n\geq 0}\left(c_{1}\lambda_1^n + c_{2} \lambda_2^n \right)(xy)^n\right]x^k.$$
Using the univariate generating function, we find
$F(x,y)=\frac{G(x,y)}{Q(x,y)},$ where
$$Q(x,y)=1-(\alpha + \gamma)xy +\alpha x^2y^2,$$
$$G(x,y)= \sum_{k\geq 0}\left[Y_{0k}+(Y_{1\,k+1}-(\alpha+\gamma)Y_{0k})xy\right] y^k$$
$$+ \sum_{k\geq 1}\left[Y_{k0}+(Y_{k+1\,1}-(\alpha+\gamma)Y_{k0})xy\right] x^k.$$

\section{Discrete multitime periodic coefficients \\Samuelson-Hicks model}

We have in mind the single-time case \cite{M} and similar ingredients with
those in Section $9$. Instead of constant parameters
$\gamma$ and $\alpha$, we use variable parameters $\gamma(t)$ and $\alpha(t)$,
with $\gamma(t)\in (0,1)$, $\alpha(t)>0$, which are supposed to be $T$-diagonal-periodic.

A {\it discrete multitime periodic coefficients Samuelson-Hicks model}
can be either the first order diagonal recurrence system
$$\begin{pmatrix}
  Y(t+{\bf 1}) \\
  C(t+{\bf 1})
 \end{pmatrix}=
 \begin{pmatrix}
  \gamma(t)+\alpha(t) & -\frac{\alpha(t)}{\gamma(t-{\bf 1})} \\
  \gamma(t) & 0
 \end{pmatrix}
 \begin{pmatrix}
 Y(t) \\
  C(t)
 \end{pmatrix}, \,\,t\in \mathbb{N}^m, $$
subject to the (initial and compatibility) conditions (\ref{h1}), (\ref{h2}), (\ref{h3}),
or the second order homogeneous diagonal recurrence equation
$$Y(t+{\bf 2})-(\gamma(t+{\bf 1})+\alpha(t+{\bf 1}))Y(t+{\bf 1})+\alpha({t+{\bf 1}}) Y(t)=0
,\,\, t\in \mathbb{N}^m,$$
subject to the (initial and compatibility) conditions (\ref{h5}), (\ref{h6}), (\ref{h7}), (\ref{h8}).

A prominent role  in the analysis of the foregoing recurrences is played by the
so-called {\it Floquet multipliers} of the second order homogeneous diagonal recurrence equation,
if its coefficients are $T$-diagonal-periodic.
We reconvert this recurrence into an equivalent diagonal recurrence system
$$\begin{pmatrix}
  Y(t+{\bf 1}) \\
  Z(t+{\bf 1})
 \end{pmatrix}=
 \begin{pmatrix}
 0&1 \\
 -\alpha(t+{\bf 1}) &\gamma({t+{\bf 1}}) + \alpha({t+{\bf 1}})
 \end{pmatrix}
 \begin{pmatrix}
 Y(t) \\
  Z(t)
 \end{pmatrix}, \,\,t\in \mathbb{N}^m.$$

The matrix of this diagonal recurrence system
$$A(t+{\bf 1})=\begin{pmatrix}
 0&1 \\
 -\alpha(t+{\bf 1}) &\gamma({t+{\bf 1}}) + \alpha({t+{\bf 1}})
 \end{pmatrix}$$
is $T$-diagonal-periodic.
A Floquet multiplier is nothing else than an eigenvalue of the matrix
$$
D(t)= \prod_{k=0}^{T-1}A \big( t+(T-k-\mu(t))\cdot {\bf 1} \big).
$$
This matrix is called {\it monodromy matrix} of the foregoing first order recurrence system.
It is a $2\times 2$- matrix with determinant
$$\Delta(t)=\det\,
D(t)= \alpha \big({t+(T-\mu(t))\cdot {\bf 1}} \big)
\cdot
\ldots \cdot \alpha \big({t+(1-\mu(t)) \cdot {\bf 1}}\big).
$$

\begin{theorem}
The Floquet multipliers of the second order homogeneous diagonal recurrence equation,
with diagonal-periodic coefficients, are the two roots of the quadratic equation
$z^{2}- z(Tr\,\,D)+ \Delta=0.$
These multipliers depend on \,$t^1-\mu(t),\ldots, t^{m}-\mu(t)$.
\end{theorem}

For the constant coefficients recurrence equation, the foregoing equation is reduced to
$z^{2}-(\gamma+\alpha)z+\alpha=0$
and the Floquet multipliers are constants.

\section{Conclusions}
This paper presents original results regarding the multivariate recurrence equations as continuation of \cite{GTU}.
Our approach to multivariate recurrence equations is advantageous for practical problems.
The original results have a great potential to solve problems in various areas such as
ecosystem dynamics, financial modeling, and economics.

Though the techniques for deriving multivariate generating functions are sometimes
paralleling the univariate theory, they achieve surprising depth in many problems.
Analytic methods for recovering coefficients of generating functions once the functions have been derived
have, however, been sorely lacking.

The authors lay no claims to the paper's being a complete presentation of all current methods
for investigation the linear multivariate recurrences with periodic coefficients. The material
presented here is a reflection of our scientific interests regarding Floquet theory.

We hope to recruit further researchers into multivariate recurrences, which still have many interesting
challenges to offer, and this explains the rather comprehensive nature of the paper.

\section*{Acknowledgments}

Partially, the work has been funded by the Sectoral Operational Programme Human Resources
Development 2007-2013 of the Ministry of European Funds through
the Financial Agreement POSDRU/159/1.5/S/132395.

Partially supported by University Politehnica of Bucharest and by Academy of Romanian Scientists.

Parts of this paper were presented at X-th International Conference
on Finsler Extensions of Relativity Theory (FERT 2014) August 18-24, 2014, Bra\c sov, Romania
and at The VIII-th International Conference ``Differential Geometry and Dynamical Systems"
(DGDS-2014), 1 - 4 September 2014, Mangalia, Romania.

Authors' Address: University Politehnica of Bucharest\\
Faculty of Applied Sciences\\
Splaiul Independentei 313\\
Bucharest 060042, Romania\\

Cristian Ghiu\\
Department of Mathematical Methods and Models\\
Email: crisghiu@yahoo.com
\vspace{0.2cm}

Raluca Tulig\u{a} (Coad\u{a}), Constantin Udri\c ste, Ionel \c Tevy\\
Department of Mathematics and Informatics\\
Emails: ralucacoada@yahoo.com; udriste@mathem.pub.ro; vascatevy@yahoo.fr

\end{document}